\documentclass[reqno, 12pt]{amsart}
\pdfoutput=1

\def\ssign{\textsection\nobreak\hspace{1pt plus 0.3pt}}
\makeatletter
\let\origsection=\section 
\def\mysection{\@mystartsection{section}{1}\z@{.7\linespacing\@plus\linespacing}{.5\linespacing}{\normalfont\scshape\centering\ssign}}
\def\section{\@ifstar{\origsection*}{\mysection}}
\def\appendix{\par\c@section\z@ \c@subsection\z@
	\let\sectionname\appendixname
	\let\section=\origsection
	\def\thesection{\@Alph\c@section}} 
\def\@mystartsection#1#2#3#4#5#6{\if@noskipsec \leavevmode \fi
	\par \@tempskipa #4\relax
	\@afterindenttrue
	\ifdim \@tempskipa <\z@ \@tempskipa -\@tempskipa \@afterindentfalse\fi
	\if@nobreak \everypar{}\else
	\addpenalty\@secpenalty\addvspace\@tempskipa\fi
	\@dblarg{\@mysect{#1}{#2}{#3}{#4}{#5}{#6}}}
\def\@mysect#1#2#3#4#5#6[#7]#8{\edef\@toclevel{\ifnum#2=\@m 0\else\number#2\fi}\ifnum #2>\c@secnumdepth \let\@secnumber\@empty
	\else \@xp\let\@xp\@secnumber\csname the#1\endcsname\fi
	\@tempskipa #5\relax
	\ifnum #2>\c@secnumdepth
	\let\@svsec\@empty
	\else
	\refstepcounter{#1}\edef\@secnumpunct{\ifdim\@tempskipa>\z@ \@ifnotempty{#8}{\@nx\enspace}\else
		\@ifempty{#8}{.}{\@nx\enspace}\fi
	}\@ifempty{#8}{\ifnum #2=\tw@ \def\@secnumfont{\bfseries}\fi}{}\protected@edef\@svsec{\ifnum#2<\@m
		\@ifundefined{#1name}{}{\ignorespaces\csname #1name\endcsname\space
		}\fi
		\@seccntformat{#1}}\fi
	\ifdim \@tempskipa>\z@ \begingroup #6\relax
	\@hangfrom{\hskip #3\relax\@svsec}{\interlinepenalty\@M #8\par}\endgroup
	\ifnum#2>\@m \else \@tocwrite{#1}{#8}\fi
	\else
	\def\@svsechd{#6\hskip #3\@svsec
		\@ifnotempty{#8}{\ignorespaces#8\unskip
			\@addpunct.}\ifnum#2>\@m \else \@tocwrite{#1}{#8}\fi
	}\fi
	\global\@nobreaktrue
	\@xsect{#5}}
\makeatother

\usepackage{amsmath,amssymb,amsthm}
\usepackage{mathrsfs}
\usepackage{mathabx}\changenotsign
\usepackage{dsfont}
\usepackage{bbm}

\usepackage{xcolor}
\usepackage[backref=section]{hyperref}
\usepackage[ocgcolorlinks]{ocgx2}
\hypersetup{
	colorlinks=true,
	linkcolor={red!60!black},
	citecolor={green!60!black},
	urlcolor={blue!60!black},
}

\usepackage[open,openlevel=2,atend]{bookmark}

\usepackage[abbrev,msc-links,backrefs]{amsrefs}
\usepackage{doi}

\renewcommand{\PrintDOI}[1]{\doi{#1}}

\usepackage[T1]{fontenc}
\usepackage{lmodern}
\usepackage[babel]{microtype}
\usepackage[english]{babel}

\usepackage{geometry}
\numberwithin{equation}{section}
\numberwithin{figure}{section}

\usepackage{enumitem}

\let\polishlcross=\l
\def\l{\ifmmode\ell\else\polishlcross\fi}

\let\emptyset=\varnothing
\let\setminus=\smallsetminus

\makeatletter
\def\moverlay{\mathpalette\mov@rlay}
\def\mov@rlay#1#2{\leavevmode\vtop{   \baselineskip\z@skip \lineskiplimit-\maxdimen
		\ialign{\hfil$\m@th#1##$\hfil\cr#2\crcr}}}
\newcommand{\charfusion}[3][\mathord]{
	#1{\ifx#1\mathop\vphantom{#2}\fi
		\mathpalette\mov@rlay{#2\cr#3}
	}
	\ifx#1\mathop\expandafter\displaylimits\fi}
\makeatother

\DeclareFontFamily{U}  {MnSymbolC}{}
\DeclareSymbolFont{MnSyC}         {U}  {MnSymbolC}{m}{n}
\DeclareFontShape{U}{MnSymbolC}{m}{n}{
	<-6>  MnSymbolC5
	<6-7>  MnSymbolC6
	<7-8>  MnSymbolC7
	<8-9>  MnSymbolC8
	<9-10> MnSymbolC9
	<10-12> MnSymbolC10
	<12->   MnSymbolC12}{}
\DeclareMathSymbol{\powerset}{\mathord}{MnSyC}{180}

\usepackage{tikz}
\usetikzlibrary{calc,decorations.pathmorphing}
\usetikzlibrary{arrows,decorations.pathreplacing}
\pgfdeclarelayer{background}
\pgfdeclarelayer{foreground}
\pgfdeclarelayer{front}
\pgfsetlayers{background,main,foreground,front}

\usepackage{multicol}
\usepackage{subcaption}
\newcommand{\pedge}[9]{
	
	\ifx\relax#6\relax
	\def\qoffs{0pt}
	\else
	\def\qoffs{#6}
	\fi
	
	\def\phedge{
		($#1+#5!\qoffs!-90:#2-#5$) -- 
		($#2+#1!\qoffs!-90:#3-#1$) -- 
		($#3+#2!\qoffs!-90:#4-#2$) -- 
		($#4+#3!\qoffs!-90:#5-#3$) -- 
		($#5+#4!\qoffs!-90:#1-#4$) -- cycle}

	\coordinate (12) at ($#1!\qoffs!90:#2$);
	\coordinate (15) at ($#1!\qoffs!-90:#5$);
	\coordinate (23) at ($#2!\qoffs!90:#3$);
	\coordinate (21) at ($#2!\qoffs!-90:#1$);
	\coordinate (34) at ($#3!\qoffs!90:#4$);
	\coordinate (32) at ($#3!\qoffs!-90:#2$);
	\coordinate (45) at ($#4!\qoffs!90:#5$);
	\coordinate (43) at ($#4!\qoffs!-90:#3$);
	\coordinate (51) at ($#5!\qoffs!90:#1$);
	\coordinate (54) at ($#5!\qoffs!-90:#4$);

	\def\nphedge{
		(15) let \p1=($(15)-#1$), \p2=($(12)-#1$) in 
		arc[start angle={atan2(\y1,\x1)}, delta angle={atan2(\y2,\x2)-atan2(\y1,\x1)-360*(atan2(\y2,\x2)-atan2(\y1,\x1)>0)}, x radius=\qoffs, y radius=\qoffs] --
		(21) let \p1=($(21)-#2$), \p2=($(23)-#2$) in 
		arc[start angle={atan2(\y1,\x1)}, delta angle={atan2(\y2,\x2)-atan2(\y1,\x1)-360*(atan2(\y2,\x2)-atan2(\y1,\x1)>0)}, x radius=\qoffs, y radius=\qoffs] --
		(32) let \p1=($(32)-#3$), \p2=($(34)-#3$) in 
		arc[start  angle={atan2(\y1,\x1)}, delta angle={atan2(\y2,\x2)-atan2(\y1,\x1)-360*(atan2(\y2,\x2)-atan2(\y1,\x1)>0)}, x radius=\qoffs, y radius=\qoffs] --
		(43) let \p1=($(43)-#4$), \p2=($(45)-#4$) in 
		arc[start angle={atan2(\y1,\x1)}, delta angle={atan2(\y2,\x2)-atan2(\y1,\x1)-360*(atan2(\y2,\x2)-atan2(\y1,\x1)>0)}, x radius=\qoffs, y radius=\qoffs] --
		(54) let \p1=($(54)-#5$), \p2=($(51)-#5$) in 
		arc[start angle={atan2(\y1,\x1)}, delta angle={atan2(\y2,\x2)-atan2(\y1,\x1)-360*(atan2(\y2,\x2)-atan2(\y1,\x1)>0)}, x radius=\qoffs, y radius=\qoffs] --
		cycle}

	\ifx\relax#7\relax
	\def\plwidth{1pt}
	\else
	\def\plwidth{#7}
	\fi
	
	\ifx\relax#9\relax
	\fill \nphedge;
	\else
	\fill[#9]\nphedge;
	\fi
	
	\ifx\relax#8\relax
	\draw[line width=\plwidth,rounded corners=\qoffs]\nphedge;
	\else
	\draw[line width=\plwidth,#8]\nphedge;
	\fi
}

\newcommand{\qedge}[7]{
	
	\ifx\relax#4\relax
	\def\qoffs{0pt}
	\else
	\def\qoffs{#4}
	\fi
	
	\def\qhedge{
		($#1+#3!\qoffs!-90:#2-#3$) --
		($#2+#1!\qoffs!-90:#3-#1$) --
		($#3+#2!\qoffs!-90:#1-#2$) -- cycle}

	\coordinate (12) at ($#1!\qoffs!90:#2$);
	\coordinate (13) at ($#1!\qoffs!-90:#3$);
	\coordinate (23) at ($#2!\qoffs!90:#3$);
	\coordinate (21) at ($#2!\qoffs!-90:#1$);
	\coordinate (31) at ($#3!\qoffs!90:#1$);
	\coordinate (32) at ($#3!\qoffs!-90:#2$);
	
	\def\nqhedge{
		(13) let \p1=($(13)-#1$), \p2=($(12)-#1$) in
		arc[start angle={atan2(\y1,\x1)}, delta angle={atan2(\y2,\x2)-atan2(\y1,\x1)-360*(atan2(\y2,\x2)-atan2(\y1,\x1)>0)}, x radius=\qoffs, y radius=\qoffs] --
		(21) let \p1=($(21)-#2$), \p2=($(23)-#2$) in
		arc[start angle={atan2(\y1,\x1)}, delta angle={atan2(\y2,\x2)-atan2(\y1,\x1)-360*(atan2(\y2,\x2)-atan2(\y1,\x1)>0)}, x radius=\qoffs, y radius=\qoffs] --
		(32) let \p1=($(32)-#3$), \p2=($(31)-#3$) in
		arc[start angle={atan2(\y1,\x1)}, delta angle={atan2(\y2,\x2)-atan2(\y1,\x1)-360*(atan2(\y2,\x2)-atan2(\y1,\x1)>0)}, x radius=\qoffs, y radius=\qoffs] --
		cycle}
	
	\ifx\relax#5\relax
	\def\qlwidth{1pt}
	\else
	\def\qlwidth{#5}
	\fi
	
	\ifx\relax#7\relax
	\fill \nqhedge;
	\else
	\fill[#7]\nqhedge;
	\fi
	
	\ifx\relax#6\relax
	\draw[line width=\qlwidth,rounded corners=\qoffs]\nqhedge;
	\else
	\draw[line width=\qlwidth,#6]\nqhedge;
	\fi
}

\newcommand{\redge}[8]{
	
	\ifx\relax#5\relax
	\def\qoffs{0pt}
	\else
	\def\qoffs{#5}
	\fi
	
	\def\rhedge{
		($#1+#4!\qoffs!-90:#2-#4$) -- 
		($#2+#1!\qoffs!-90:#3-#1$) -- 
		($#3+#2!\qoffs!-90:#4-#2$) -- 
		($#4+#3!\qoffs!-90:#1-#3$) -- cycle}

	\coordinate (12) at ($#1!\qoffs!90:#2$);
	\coordinate (14) at ($#1!\qoffs!-90:#4$);
	\coordinate (23) at ($#2!\qoffs!90:#3$);
	\coordinate (21) at ($#2!\qoffs!-90:#1$);
	\coordinate (34) at ($#3!\qoffs!90:#4$);
	\coordinate (32) at ($#3!\qoffs!-90:#2$);
	\coordinate (41) at ($#4!\qoffs!90:#1$);
	\coordinate (43) at ($#4!\qoffs!-90:#3$);
	
	\def\nrhedge{
		(14) let \p1=($(14)-#1$), \p2=($(12)-#1$) in 
		arc[start angle={atan2(\y1,\x1)}, delta angle={atan2(\y2,\x2)-atan2(\y1,\x1)-360*(atan2(\y2,\x2)-atan2(\y1,\x1)>0)}, x radius=\qoffs, y radius=\qoffs] --
		(21) let \p1=($(21)-#2$), \p2=($(23)-#2$) in 
		arc[start angle={atan2(\y1,\x1)}, delta angle={atan2(\y2,\x2)-atan2(\y1,\x1)-360*(atan2(\y2,\x2)-atan2(\y1,\x1)>0)}, x radius=\qoffs, y radius=\qoffs] --
		(32) let \p1=($(32)-#3$), \p2=($(34)-#3$) in 
		arc[start angle={atan2(\y1,\x1)}, delta angle={atan2(\y2,\x2)-atan2(\y1,\x1)-360*(atan2(\y2,\x2)-atan2(\y1,\x1)>0)}, x radius=\qoffs, y radius=\qoffs] --
		(43) let \p1=($(43)-#4$), \p2=($(41)-#4$) in 
		arc[start angle={atan2(\y1,\x1)}, delta angle={atan2(\y2,\x2)-atan2(\y1,\x1)-360*(atan2(\y2,\x2)-atan2(\y1,\x1)>0)}, x radius=\qoffs, y radius=\qoffs] --
		cycle}
	
	\ifx\relax#6\relax
	\def\rlwidth{1pt}
	\else
	\def\rlwidth{#6}
	\fi
	
	\ifx\relax#8\relax
	\fill \nrhedge;
	\else
	\fill[#8]\nrhedge;
	\fi
	
	\ifx\relax#7\relax
	\draw[line width=\rlwidth,rounded corners=\qoffs]\nrhedge;
	\else
	\draw[line width=\rlwidth,#7]\nrhedge;
	\fi
}

\let\epsilon=\varepsilon

\let\rho=\varrho
\let\theta=\vartheta

\newcommand{\cF}{\mathcal{F}}

\newtheoremstyle{note}  {4pt}  {4pt}  {\sl}  {}  {\bfseries}  {.}  {.5em}          {}
\newtheoremstyle{introthms}  {3pt}  {3pt}  {\itshape}  {}  {\bfseries}  {.}  {.5em}          {\thmnote{#3}}
\newtheoremstyle{remark}  {2pt}  {2pt}  {\rm}  {}  {\bfseries}  {.}  {.3em}          {}

\theoremstyle{plain}
\newtheorem{theorem}{Theorem}[section]

\newtheorem{claim}[theorem]{Claim}

\theoremstyle{note}

\theoremstyle{remark}

\usepackage{lineno}
\newcommand*\patchAmsMathEnvironmentForLineno[1]{
	\expandafter\let\csname old#1\expandafter\endcsname\csname #1\endcsname
	\expandafter\let\csname oldend#1\expandafter\endcsname\csname end#1\endcsname
	\renewenvironment{#1}
	{\linenomath\csname old#1\endcsname}
	{\csname oldend#1\endcsname\endlinenomath}}
\newcommand*\patchBothAmsMathEnvironmentsForLineno[1]{
	\patchAmsMathEnvironmentForLineno{#1}
	\patchAmsMathEnvironmentForLineno{#1*}}
\AtBeginDocument{
	\patchBothAmsMathEnvironmentsForLineno{equation}
	\patchBothAmsMathEnvironmentsForLineno{align}
	\patchBothAmsMathEnvironmentsForLineno{flalign}
	\patchBothAmsMathEnvironmentsForLineno{alignat}
	\patchBothAmsMathEnvironmentsForLineno{gather}
	\patchBothAmsMathEnvironmentsForLineno{multline}
}

\def\ex{\text{\rm ex}}

\usepackage{scalerel}

\newsavebox\vdegbox
\savebox\vdegbox{\tikz{
		\draw[black,fill=black] (90:1) circle (.35);
		\draw[black,line width=0.10cm] (210:1) circle (.30);
		\draw[black,line width=0.10cm] (330:1) circle (.30);
		\draw[opacity=0] (0:1.2) circle (0.1);
}}

\newsavebox\vvbox
\savebox\vvbox{\tikz{
		\draw[black,line width=0.10cm] (90:1) circle (.30);
		\draw[black,fill=black] (210:1) circle (.35);
		\draw[black,fill=black] (330:1) circle (.35);
		\draw[opacity=0] (0:1.2) circle (0.1);
}}

\newsavebox\pdegbox
\savebox\pdegbox{\tikz{
		\draw[black,line width=0.10cm] (90:1) circle (.30);
		\draw[black,fill=black] (210:1) circle (.35);
		\draw[black,fill=black] (330:1) circle (.35);
		\draw[black,line width=0.28cm ] (210:1) -- (330:1);
		\draw[opacity=0] (0:1.2) circle (0.1);
}}

\newsavebox\vvvbox
\savebox\vvvbox{\tikz{
		\draw[black,fill=black] (90:1) circle (.35);
		\draw[black,fill=black] (210:1) circle (.35);
		\draw[black,fill=black] (330:1) circle (.35);
		\draw[opacity=0] (0:1.2) circle (0.1);
}}

\newsavebox\evbox
\savebox\evbox{\tikz{
		\draw[black,fill=black] (90:1) circle (.35);
		\draw[black,fill=black] (210:1) circle (.35);
		\draw[black,fill=black] (330:1) circle (.35);
		\draw[black,line width=0.28cm ] (210:1) -- (330:1);
		\draw[opacity=0] (0:1.2) circle (0.1);
}}

\newsavebox\eebox
\savebox\eebox{\tikz{
		\draw[black,fill=black] (90:1) circle (.35);
		\draw[black,fill=black] (210:1) circle (.35);
		\draw[black,fill=black] (330:1) circle (.35);
		\draw[black,line width=0.28cm ] (90:1) -- (330:1);
		\draw[black,line width=0.28cm ] (90:1) -- (210:1);
		\draw[opacity=0] (0:1.2) circle (0.1);
}}

\newsavebox\eeebox
\savebox\eeebox{\tikz{
		\draw[black,fill=black] (90:1) circle (.35);
		\draw[black,fill=black] (210:1) circle (.35);
		\draw[black,fill=black] (330:1) circle (.35);
		\draw[black,line width=0.28cm ] (90:1) -- (330:1);
		\draw[black,line width=0.28cm ] (90:1) -- (210:1);
		\draw[black,line width=0.28cm ] (210:1) -- (330:1);
		\draw[opacity=0] (0:1.2) circle (0.1);
}}

\makeatletter
\newcommand{\overrighharpoonup}[1]{\ThisStyle{%
		\vbox {\m@th\ialign{##\crcr
				\rightharpoonupfill \crcr
				\noalign{\kern-\p@\nointerlineskip}
				$\hfil\SavedStyle#1\hfil$\crcr}}}}

\def\rightharpoonupfill{%
	$\SavedStyle\m@th\mkern+0.8mu\cleaders\hbox{$\shortbar\mkern-4mu$}\hfill\rightharpoonuptip\mkern+0.8mu$}

\def\rightharpoonuptip{%
	\raisebox{\z@}[2pt][1pt]{\scalebox{0.55}{$\SavedStyle\rightharpoonup$}}}

\def\shortbar{%
	\smash{\scalebox{0.55}{$\SavedStyle\relbar$}}}
\makeatother

\makeatletter
\newcommand{\overlefharpoonup}[1]{\ThisStyle{%
		\vbox {\m@th\ialign{##\crcr
				\leftharpoonupfill \crcr
				\noalign{\kern-\p@\nointerlineskip}
				$\hfil\SavedStyle#1\hfil$\crcr}}}}

\def\leftharpoonupfill{%
	$\SavedStyle\m@th\mkern+0.8mu\cleaders\hbox{$\shortbar\mkern-4mu$}\hfill\leftharpoonuptip\mkern+0.8mu$}

\def\leftharpoonuptip{%
	\raisebox{\z@}[2pt][1pt]{\scalebox{0.55}{$\SavedStyle\leftharpoonup$}}}

\makeatletter
\newsavebox\myboxA
\newsavebox\myboxB
\newlength\mylenA

\newcommand*\xoverline[2][0.75]{%
	\sbox{\myboxA}{$\m@th#2$}%
	\setbox\myboxB\null
	\ht\myboxB=\ht\myboxA%
	\dp\myboxB=\dp\myboxA%
	\wd\myboxB=#1\wd\myboxA
	\sbox\myboxB{$\m@th\overline{\copy\myboxB}$}
	\setlength\mylenA{\the\wd\myboxA}
	\addtolength\mylenA{-\the\wd\myboxB}%
	\ifdim\wd\myboxB<\wd\myboxA%
	\rlap{\hskip 0.5\mylenA\usebox\myboxB}{\usebox\myboxA}%
	\else
	\hskip -0.5\mylenA\rlap{\usebox\myboxA}{\hskip 0.5\mylenA\usebox\myboxB}%
	\fi}
\makeatother

\begin{document}
	
    \title[Upper bounds on the running time of bootstrap percolation]{Upper bounds on the running time of bootstrap percolation}

	\author[Weichan Liu]{Weichan Liu}
	\address{School of Mathematics, Shandong University, Jinan, China}
	\email{wcliu@sdu.edu.cn}

    \author[Xiangxiang Nie]{Xiangxiang Nie}
	\address{Data Science Institute, Shandong University, Jinan, China}
	\email{xiangxiangnie@sdu.edu.cn}

    \author[Sim\'on Piga]{Sim\'on Piga}
	\address{Institute of Computer Science of the Czech Academy of Sciences}
	\email{piga@cs.cas.cz}
    
	\author[Bjarne Sch\"{u}lke]{Bjarne Sch\"{u}lke}
	\address{Extremal Combinatorics and Probability Group, Institute for Basic Science, Daejeon, South Korea}
	\email{schuelke@ibs.re.kr}

	\begin{abstract}
        For~$k$-graphs~$F$ and~$H_0$ the~$F$-bootstrap percolation process (or~$F$-process) starting with~$H_0$ is a sequence~$(H_i)_{i\geq0}$ of~$k$-graphs such that~$H_{i+1}$ is obtained from~$H_i$ by adding all those~$e\in V(H_0)^{(k)}\setminus E(H_i)$ as edges that complete a new copy of~$F$.
        The running time of this~$F$-process, denoted by~$M_F(H_0)$, is the smallest~$i$ with~$H_i=H_{i+1}$.
        Bollob\'as proposed the problem of determining the maximum running time for~$n\in\mathds{N}$, i.e.,~$M_F(n)=\max_{\vert V(H_0)\vert=n}M_F(H_0)$.
        Although this problem has received a lot of attention recently, until now the best known upper bound for~$M_{K_t}(n)$, with~$t\geq5$, was the trivial bound of essentially~$\binom{n}{2}$.

        Here we provide the first non-trivial upper bound for this problem by showing that $$M_{K_t}(n)\leq\Big(\frac{t-3}{t-2}+o(1)\Big)\binom{n}{2}$$ holds for every integer~$t\geq 3$.
        In fact, we prove the following more general result.
        For every~$k\geq2$, every~$k$-graph~$F$, and every~$e\in E(F)$ we have~$M_F(n)\leq\big(\pi(F-e)+o(1)\big)\binom{n}{k}$, where~$\pi$ is the Tur\'an density.
	\end{abstract}
	
	\maketitle
	\section{Introduction}
	
	A \emph{$k$-uniform hypergraph} (or \emph{$k$-graph}) $H$ consists of a vertex set $V(H)$ and a set of edges~$E(H) \subseteq V(H)^{(k)}$, where~$V(H)^{(k)} = \{ S \subseteq V(H) : |S| = k \}$.
	Given $k$-graphs $F$ and~$H$, a copy of~$F$ in~$H$ is a subhypergraph of~$H$ that is isomorphic to~$F$.
    The number of copies of~$F$ in~$H$ is denoted by~$c(F,H)$.
    The \emph{$F$-bootstrap percolation process} (or \emph{$F$-process}) \emph{starting with}~$H$ is the sequence~$(H_i)_{i\geq0}$ of~$k$-graphs with~$H_i$ defined inductively as follows.
    Set~$H_0=H$ and, given~$H_i$, let~$H_{i+1}$ be the~$k$-graph with vertex set~$V(H)$ and edge set
    \begin{align*}
		E(H_{i+1})=E(H_i)\cup\big\{e\in V(H)^{(k)}:\: c(F,H_i\cup e)>c(F,H_i)\big\}\,,
	\end{align*}
    where~$H_i\cup e$ denotes the~$k$-graph~$\big(V(H),E(H_i)\cup\{e\}\big)$.

    This notion originates in the work of Bollob\'as~\cite{B:68}, where it arose in the context of weak saturation.
    Since then, connections to other areas such as cellular automata, more general bootstrap-type dynamics~\cite{M:17}, and statistical physics (see~\cite{BBM:12}~and~\cite{CLR:79}) have been established.
    Bollob\'as (see~\cite{BPRS:17}) suggested the problem of determining the maximum running time of such processes (a random variant was considered in~\cite{GKP:17}).
    More precisely, the running time of the~$F$-process starting with~$H$ is~$M_F(H)=\min\{t\in\mathbb{N}_{\geq0}:\:H_t = H_{t + 1}\}$.
    Given~$n\in\mathds{N}$ (and for~$k=2$), Bollob\'as asked to determine
	\[
	M_F(n) = \max \big\{ M_F(H) :\: H \text{ is a } k\text{-graph on } n \text{ vertices} \big\}\,.
	\]
    For every~$k$-graph~$F$ one readily observes a trivial upper bound of~$\binom{n}{k}$ for~$M_F(n)$ by noticing that every~$k$-subset of vertices can be added as an edge at most once.
    This observation can be pushed slightly further since some edges need to be present in~$H_0$ for the process not to stop immediately.
    However, these considerations do not give any better upper bound than~$(1-o(1))\binom{n}{k}$.
    
    The maximum running time of~$K_t$-processes has gathered particular attention. 
    Bollob\'as, Przykucki, Riordan, and Sahasrabudhe~\cite{BPRS:17} and, independently, Matzke~\cite{M:15} showed that $M_{K_4}(n)=n-3$ for all $n\ge 3$. 
    In~\cite{BPRS:17} they also proved that~$M_{K_t}(n)\ge n^{2-\lambda_t-o(1)}$ holds for~$t\ge 5$, where~$\lambda_t$ is some explicit constant with~$\lambda_t\rightarrow 0$ as~$t\rightarrow \infty$.
    Although they did not find any non-trivial upper bound, they conjectured that a subquadratic upper bound should hold as well, i.e., that~$M_{K_t}(n)=o(n^2)$.
    Balogh, Kronenberg, Pokrovskiy, and Szab\'o~\cite{BKPS:26} disproved this ambitious upper bound by showing that~$M_{K_t}(n)=\Omega(n^2)$ for all~$t\ge 6$.
    However, they did not obtain any non-trivial upper bound either, and so the exact asymptotics of~$M_{K_t}(n)$ remained elusive.
    For~$t=5$ they established~$M_{K_5}(n)\geq n^{2-o(1)}$ by modifying Behrend's construction for arithmetic progressions, but the original conjecture remains open in this case.

    Despite the aforementioned upper bound conjecture by Bollob\'as, Przykucki, Riordan, and Sahasrabudhe~\cite{BPRS:17} and all the subsequent works on the maximum running time of bootstrap percolation processes (see~\cite{FMS:26} for an overview), no non-trivial upper bound has been found for~$M_{K_t}(n)$ with~$t\geq 5$.
    This frustrating state of affairs has been pointed out repeatedly in the literature, for instance in~\cites{BPRS:17,BKPS:26,EJKL:24}, and most recently in the survey~\cite{FMS:26}.
    
    Here, we establish the first non-trivial upper bound on~$M_{K_t}(n)$ for~$t\geq5$.
    


    \begin{theorem}\label{cor:graphcliques}
        For every~$t\geq 3$ we have~$M_{K_t}(n)\leq\big(\frac{t-3}{t-2}+o(1)\big)\binom{n}{2}$.
    \end{theorem}

    The~$o(1)$ notation in the statement means that for every~$\varepsilon>0$ there is some~$n_0$ such that~$M_{K_t}(n)\leq\big(\frac{t-3}{t-2}+\varepsilon\big)\binom{n}{2}$ for all~$n\in\mathds{N}$ with~$n\geq n_0$.
    
    Moreover, we provide a non-trivial upper bound for every graph in terms of the \textit{chromatic number}.
    Recall that for a graph~$F$ the chromatic number~$\chi(F)$ denotes the smallest~$c\in\mathds{N}$ such that there is a proper colouring of~$F$ with~$c$ colours.

    \begin{theorem}\label{thm:graphs}
        For every graph~$F$ and every~$e\in E(F)$ we have 
        \begin{equation*}
         M_F(n)\leq\bigg(\frac{\chi(F-e)-2}{\chi(F-e)-1}+o(1)\bigg)\binom{n}{2}\,.
        \end{equation*}
	\end{theorem}
    Note that Theorem~\ref{thm:graphs} directly implies Theorem~\ref{cor:graphcliques}.

    Apart from few exceptions (see in particular the work by Fabian, Morris, and Szab\'o in~\cite{FMS:25} and their recent survey~\cite{FMS:26}), for a general graph~$F$ no non-trivial upper bound on~$M_F(n)$ was known.
    One notable exception is a nice result by Fabian, Morris, and Szab\'o (Proposition 1.6 in~\cite{FMS:2508}), which provides a non-trivial upper bound for every bipartite graph.
    
	Recently, the investigation of~$F$-processes and their maximum running times was initiated for hypergraphs, see~\cite{NR:23},~\cite{HL:24}, and~\cite{EJKL:24}.
    While these works established that~$M_{K^{(k)}_t}(n)=\Omega(n^k)$ holds for all~$k\geq3$ and~$t\geq k+1$, also in this setting non-trivial upper bounds were lacking.
    The only non-trivial upper bounds for hypergraphs (of uniformity at least~$3$) before this work were for~$K_4^{(3)-}$, proved in~\cite{EJKL:24}, and for extension hypergraphs~\cite{LSZ:26}.
    In particular, so far the best known upper bound for any~$k$-uniform clique was the trivial upper bound of essentially~$\binom{n}{k}$.
    Espuny D\'iaz, Janzer, Kronenberg, and Lada~\cite{EJKL:24} proposed the problem of finding non-trivial upper bounds for~$M_{K^{(k)}_t}(n)$ as a step towards understanding the asymptotics of~$M_{K^{(k)}_t}(n)/\binom{n}{k}$ for~$n\to\infty$.
    
    We provide the first non-trivial upper bounds in this more general setup as well and indeed Theorem~\ref{thm:graphs} is a simple corollary of our general result for hypergraphs.
    To state it, we briefly recall the following definitions.
    Given a~$k$-graph~$F$ and~$n\in\mathds{N}$, the extremal number of~$n$ and~$F$, denoted by~$\ex(n,F)$, is the maximum number of edges that a~$k$-graph on~$n$ vertices can have without containing a copy of~$F$.
    The Tur\'an density of a~$k$-graph~$F$ is~$\pi(F)=\displaystyle\lim_{n\to\infty}\tfrac{\ex(n,F)}{\tbinom{n}{k}}$; this limit exists due to a simple monotonicity argument~\cite{KNS:64}.
    For the sake of brevity we omit an extensive discussion of this rich area and instead refer the reader to the surveys by F\"uredi~\cite{F:91}, Sidorenko~\cite{S:95}, and Keevash~\cite{K:11} for more background.
    
    We obtain the following general upper bound on the maximum running time of~$F$-processes in terms of the Tur\'an density.

    \begin{theorem}\label{thm:main}
        For every integer~$k\geq2$, every~$k$-graph~$F$, and every~$e\in E(F)$ we have 
        \begin{equation*}
         M_F(n)
         \leq
         \big(\pi(F-e)+o(1)\big)\binom{n}{k}\,.
        \end{equation*}
    \end{theorem}

    

   From this, Theorem~\ref{thm:graphs} follows easily by the Erd\H{o}s-Stone-Simonovits theorem (see~\cite{ES:46} and~\cite{ES:66}), which states that for every graph~$F$ we have~$\pi(F)=\frac{\chi(F)-2}{\chi(F)-1}$.
   Somewhat surprisingly, despite the generality of Theorem~\ref{thm:main} and the previous elusiveness of upper bounds on the running time of bootstrap percolation processes, the proof of Theorem~\ref{thm:main} is quite short and does not rely on any heavy machinery.

    \section{Preliminaries}\label{sec:prelim}

        Given~$i\in\mathds{N}$, we write~$[i]=\{1,\dots,i\}$.
        For a~$k$-graph~$H$ and~$e\in V(H)^{(k)}$ denote by~$H \cup e$ the~$k$-graph~$(V(H), E(H)\cup\{e\})$ and set~$v(H)=|V(H)|$.
        For~$e\in E(H)$ we write~$H-e$ for the~$k$-graph~$\big(V(H),E(H)\setminus \{e\}\big)$.
        Usually, we omit parentheses around singletons.
        
        We use the following fundamental result, often referred to as supersaturation (see, for instance results in~\cite[Lemma 2.1 and Theorem 2.2]{K:11}).
        \begin{theorem}[Supersaturation]\label{thm:supersaturation}
            \item\label{it:supsat:original} For every~$k$-graph~$F$ and~$\varepsilon>0$, there are~$\delta>0$ and~$n_0$ such that every~$k$-graph on~$n\geq n_0$ vertices with at least~$(\pi(F)+\varepsilon)\binom{n}{k}$ edges contains at least~$\delta n^{v(F)}$ copies of~$F$.
        \end{theorem}

    \section{Proof of Theorem~\ref{thm:main}}
        The following is a weaker but simpler version of Theorem~\ref{thm:main}.
        \begin{theorem}\label{thm:simple}
            For every integer~$k\geq2$ and every~$k$-graph~$F$ we have $$M_F(n)\leq\big(\pi(F)+o(1)\big)\binom{n}{k}\,.$$
        \end{theorem}

        Although Theorem~\ref{thm:simple} follows from Theorem~\ref{thm:main}, we give a direct proof here since it illustrates the main idea behind the proof of Theorem~\ref{thm:main} particularly clearly.
        \begin{proof}[Proof of Theorem~\ref{thm:simple}]
            Given~$\varepsilon>0$, an integer~$k\geq2$, and a~$k$-graph~$F$, choose~$\delta>0$ and~$n\in\mathds{N}$ such that $$\varepsilon,k,v(F)\gg\delta\gg n^{-1}\,.$$
            We need to show that~$M_F(n)\leq\big(\pi(F)+\varepsilon\big)\binom{n}{k}$.
            Note that if~$F$ has only one edge, the result is trivial, so assume that~$F$ has at least two edges.
            Let~$(H_i)_{i\geq0}$ be an~$F$-process starting with~$H=(V,E)$, where~$\vert V\vert=n$, that satisfies~$M_F(H)= M_F(n)=:\tau$.
            For each~$i\in[\tau]$ choose~$e_i\in E(H_i)\setminus E(H_{i-1})$ arbitrarily.
            Let~$G$ be the graph on~$V$ whose edge set is~$\{e_i:\:i\in[\tau]\}$.
            Assume that~$\tau\geq(\pi(F)+\varepsilon)\binom{n}{k}$.
            We aim to derive a contradiction by establishing upper and lower bounds for the number of copies of~$F$ in~$G$.
            Denote the set of copies of~$F$ in~$G$ by~$\cF$.
            
            On the one hand, supersaturation and the hierarchy of constants imply that $$\vert\cF\vert\geq\delta n^{v(F)}\,.$$
            For any~$D\in\cF$ let~$i(D)=\max\{i\in[\tau]:\:e_i\in E(D)\}$.
            Now observe that for a copy~$D$ of~$F$ in~$G$ we must have~$e_{i(D)-1}\in E(D)$.
            Otherwise, by the definition of~$G$, for every edge~$e_i$ of~$D-e_{i(D)}$ we have~$i\leq i(D)-2$ and so~$D-e_{i(D)}\subseteq H_{i(D)-2}$.
            Therefore, the definition of an~$F$-process implies that~$e_{i(D)}\in E(H_{i(D)-1})$, a contradiction.
            Thus, for every~$e_j\in E(G)$, with~$j\geq2$ and~$v\in e_{j-1}\setminus e_j$ we know that~$v$ is contained in every copy~$D\in\cF$ with~$e_{i(D)}=e_j$.
            This entails that there are at least~$k+1$ vertices (namely~$v$ and the vertices in~$e_j$) which are contained in every copy~$D\in\cF$ with~$e_{i(D)}=e_j$.
            Since of course no~$D\in\cF$ satisfies~$i(D)=1$, we conclude that for every~$e\in E(G)$ it holds that~$\vert\{D\in\cF:\:e=e_{i(D)}\}\vert\leq n^{v(F)-k-1}$.
            Hence, we obtain
            \begin{align*}
                \vert\cF\vert=\sum_{e\in E(G)}\vert\{D\in\cF:\:e=e_{i(D)}\}\vert\leq n^k\cdot n^{v(F)-k-1}=n^{v(F)-1}\,,
            \end{align*}
            contradicting the lower bound above because of the hierarchy.
        \end{proof}
        
    \begin{proof}[Proof of Theorem~\ref{thm:main}]
        Given~$\varepsilon>0$, an integer~$k\geq2$, and a~$k$-graph~$F$, choose~$\delta>0$ and~$n\in\mathds{N}$ such that $$\varepsilon,k,v(F)\gg\delta\gg n^{-1}\,.$$
        Let~$e\in E(F)$ and set~$F^-=F-e$.
        We need to show that~$M_F(n)
         \leq
         \big(\pi(F^-)+\varepsilon\big)\binom{n}{k}$.
        One can observe that the statement holds if~$F$ has at most two edges, so assume that~$F$ has at least three edges.
        Let~$(H_i)_{i\geq0}$ be an~$F$-process starting with~$H=(V,E)$, where~$\vert V\vert=n$, that satisfies~$M_F(H)= M_F(n)=:\tau$.
        For each~$i\in[\tau]$ choose~$e_i\in E(H_i)\setminus E(H_{i-1})$ arbitrarily.
        Let~$G$ be the graph on~$V$ whose edge set is~$\{e_i:\:i\in[\tau]\}$.
        Assume that~$\tau\geq(\pi(F^-)+\varepsilon)\binom{n}{k}$.
        We aim to derive a contradiction by establishing upper and lower bounds for the number of copies of~$F^-$ in~$G$.
        Denote the set of copies of~$F^-$ in~$G$ by~$\cF$.
        Supersaturation and the hierarchy of constants imply that
        \begin{align}\label{eq:Flower}
            \vert\cF\vert\geq\delta n^{v(F^-)}=\delta n^{v(F)}\,.
        \end{align}
        For every~$D\in\cF$ define~$i(D)=\max\{i\in[\tau]:\:e_i\in E(D)\}$ and~$j(D)=\max\{j\in[\tau]:\:e_j\in E(D)\setminus e_{i(D)}\}$.
        Furthermore, for every~$D\in\cF$ fix some~$e_D\in V^{(k)}$ such that~$D \cup e_D$ is a copy of~$F$ (if there are several choices, pick one arbitrarily).
        For~$e\in E(H_{\tau})$ let~$\mathfrak{s}(e)$ denote the step in which~$e$ is added, i.e.,~$\mathfrak{s}(e)=\min\{s\in[\tau]:\:e\in E(H_s)\setminus E(H_{s-1})\}$ if~$e\notin E(H_0)$ and~$\mathfrak{s}(e)=0$ if~$e\in E(H_0)$.
        Depending on how~$\mathfrak{s}(e_D)$ relates to~$j(D)$ and~$i(D)$, we assign to every~$D \in \cF$ one of the following four types.
        \begin{enumerate}[leftmargin=1.95cm]
            \item[Type~$1$:]\label{it:type1}
            $\mathfrak{s}(e_D)\leq j(D)$,
            \item[Type~$2$:]\label{it:type2}
            $j(D)< \mathfrak{s}(e_D)< i(D)$,
            \item[Type~$3$:]\label{it:type3}
            $i(D) < \mathfrak{s}(e_D)$,
            \item[Type~$4$:]\label{it:type4}
            $\mathfrak{s}(e_D)=i(D)$.
        \end{enumerate}

    Let~$\cF_{\alpha}$ be the family of copies~$D\in \cF$ of Type~$\alpha$ for~$\alpha\in[4]$.
    Further, for~$i\in[\tau]$ let~$\cF_{\alpha}^i=\{D\in\cF_{\alpha}:\:i(D)=i\}$.
    Next we find upper bounds on~$\vert\cF_{\alpha}\vert$ for every~$\alpha\in[4]$.

    \begin{claim}\label{cl:F1}
        $\vert \cF_1\vert\leq n^{v(F)-1}$.
    \end{claim}

    \begin{proof}
        First observe that for~$D \in \cF_1$ we must have~$e_{j(D)} = e_{i(D)-1}$.
        Indeed, from the definition of Type~$1$ and~$G$ it follows that in~$H_{j(D)}\cup e_{i(D)}$ there is a copy of~$F$ containing~$e_{i(D)}$.
        Hence,~$e_{i(D)}\in E(H_{j(D)+1})$.
        Due to the definition of~$G$, this means that~$i(D)=j(D)+1$.
        
        Now fix~$j\in[\tau]$ with~$j\geq2$ and let~$v \in e_{j-1} \setminus e_j$. 
        By the above, every~$D \in \mathcal{F}_1^j$ must contain~$v$. 
        In particular, every~$D\in\cF_1^j$ contains the~$k+1$ vertices in the set~$e_j\cup v$.
        Since no~$D\in\cF$ satisfies~$i(D)=1$, it thus follows that for every~$j\in[\tau]$ we have~$
        \vert\cF_1^j\vert\leq n^{v(F)-(k+1)}$. 
        Summing over all~$j\in[\tau]$, we obtain
        \begin{align*}
            \vert\mathcal{F}_1\vert=\sum_{j\in[\tau]} \vert\cF_1^j\vert\leq \binom{n}{k} \cdot n^{v(F)-k-1}= n^{v(F)-1}\,,
        \end{align*}
        as desired.
    \end{proof}
    
    To show the remaining upper bounds, we first observe the following for~$\alpha\in\{2,3,4\}$.
    Since for every~$D\in\cF_{\alpha}$, it holds that~$|e_D \cup e_{i(D)}|\geq k+1$ and~$e_D \cup e_{i(D)}\subseteq V(D)$, we have~$\vert\{D\in\cF_{\alpha}^i:\:e_D=e\}\vert\leq n^{v(F)-\vert e_D\cup e_{i(D)}\vert}\leq n^{v(F)-k-1}$ for every~$e\in V^{(k)}$.
    Further, note the trivial fact that if for some~$e\in V^{(k)}$ there is no~$D\in\cF_{\alpha}^i$ with~$e_D=e$, then~$\vert\{D\in\cF_{\alpha}^i:\:e_D=e\}\vert=0$.
    Thus, we infer
    \begin{align}\label{eq:generalupper}
        |\cF_{\alpha}|&=\sum_{i\in[\tau]}\vert\cF_{\alpha}^i\vert=\sum_{i \in [\tau]} \sum_{e\in V^{(k)}}\vert\{D\in\cF_{\alpha}^i:\:e_D=e\}\vert\nonumber\\
        &\leq\sum_{i \in [\tau]}\big\vert\{e\in V^{(k)}:\:e=e_D\text{ for some }D\in\cF_{\alpha}^i\}\big\vert \cdot n^{v(F) - (k+1)}\,.
    \end{align}

    Next we show that every~$e\in V^{(k)}$ can only be counted once in the above sum.
    \begin{claim}\label{cl:disjointsets}
        For every~$\alpha\in\{2,3,4\}$ and distinct~$i,j\in[\tau]$ we have $$\{e\in V^{(k)}:\:e=e_D\text{ for some }D\in\cF_{\alpha}^i\}\cap\{e\in V^{(k)}:\:e=e_D\text{ for some }D\in\cF_{\alpha}^j\}=\emptyset\,.$$
    \end{claim}

    \begin{proof}
        Observe that for every~$D\in\cF_2$ we must have~$\mathfrak{s}(e_D)=i(D)-1$.
        Indeed, by the definition of Type~$2$ and~$G$, we know that in~$H_{\mathfrak{s}(e_D)}\cup e_{i(D)}$ there is a copy of~$F$ containing~$e_{i(D)}$.
        Keeping in mind that~$\mathfrak{s}(e_D)<i(D)$ and the definition of the~$F$-process, we conclude~$i(D)=\mathfrak{s}(e_D)+1$.
        Therefore, once we fix~$i\in[\tau]$, we know that for every~$D\in\cF_2^i$ we must have~$\mathfrak{s}(e_D)=i-1$.
        This in turn entails that the sets~$\{e\in V^{(k)}:\:e=e_D\text{ for some }D\in\cF_2^i\}$ are disjoint for distinct~$i\in[\tau]$.
        
        Note that for every~$D\in\cF_3$ we know that~$H_{i(D)}\cup e_D$ contains a copy of~$F$ with~$e_D$ being an edge of this copy, and that~$e_D\notin E(H_{i(D)})$ by the definition of Type~$3$.
        Hence, the definition of the~$F$-process implies that~$e(D)$ is added in the next step, i.e.,~$\mathfrak{s}(e_D)=i(D)+1$.
        Therefore, once we fix~$i\in[\tau]$, we know that for every~$D\in\cF_3^i$ we must have~$\mathfrak{s}(e_D)=i+1$.
        This in turn entails that the sets~$\{e\in V^{(k)}:\:e=e_D\text{ for some }D\in\cF_3^i\}$ are disjoint for distinct~$i\in[\tau]$.
        
        For~$D\in\cF_4$, we have~$\mathfrak{s}(e_D)=i(D)$ by definition of Type~$4$.
        Therefore, once we fix~$i\in[\tau]$, we know that every~$D\in\cF_4^i$ must satisfy~$\mathfrak{s}(e_D)=i$ and the claim follows.
    \end{proof}

    Claim~\ref{cl:disjointsets} reveals that for~$\alpha\in\{2,3,4\}$ we have $$\sum_{i \in [\tau]}\big\vert\{e\in V^{(k)}:\:e=e_D\text{ for some }D\in\cF_{\alpha}^i\}\big\vert\leq\binom{n}{k}\,.$$
    Combining this with~\eqref{eq:generalupper} yields
    \begin{align}\label{eq:Falpha}
        \vert\cF_{\alpha}\vert\leq n^{v(F)-(k+1)}\cdot\binom{n}{k}\leq n^{v(F)-1}\,.
    \end{align}

    Comparing the bounds in~\eqref{eq:Falpha} and the bound in Claim~\ref{cl:F1} with~\eqref{eq:Flower} entails
    $$\delta n^{v(F)}\leq\vert\cF\vert\leq\sum_{\alpha\in[4]}\vert\cF_{\alpha}\vert\leq 5\cdot n^{v(F)-1}\,.$$
    Due to the chosen hierarchy, this is a contradiction.
    \end{proof}

    \section*{Acknowledgements}
	The research was supported by the Postdoctoral Fellowship Program of CPSF under Grant Number GZC20252020, the China Postdoctoral Science Foundation 2025M783118 (Weichan Liu), by RVO:67985807 and the Czech Science Foundation Project 26-23695S (Simón Piga), and by the Young Scientist Fellowship IBS-R029-Y7 (Bjarne Sch\"ulke).
    
    No AI tools have been used for any aspects of this work.

\end{document}